\documentclass[a4paper, notitlepage]{article}
\usepackage[utf8]{inputenc}
\usepackage[T1]{fontenc}
\usepackage{lmodern,microtype}
\usepackage{xparse}
\usepackage[british]{babel}
\usepackage[top=4cm,bottom=4cm,left=4.5cm,right=4.5cm]{geometry}

\usepackage{calc,amssymb,mathtools,amsthm}\mathtoolsset{mathic}
\usepackage{stmaryrd,units,paralist,tabularx,booktabs,multicol,mathbbol}
\usepackage{pgf,tikz,tikz-cd}

\usepackage[alphabetic,non-sorted-cites,sorted]{amsrefs}

\usepackage[hidelinks,pdftex]{hyperref}
\usepackage{cleveref}
\newtheorem{thm}{Theorem}\crefname{thm}{theorem}{theorems}
\newtheorem*{thm*}{Theorem}
\newtheorem{defn}[thm]{Definition}\crefname{defn}{definition}{definition}
\newtheorem*{defn*}{Definition}
\crefname{conj}{conjecture}{conjectures}
\newtheorem*{conj*}{Theorem}
\crefname{quest}{question}{questions}
\newtheorem*{quest*}{Question}
\newtheorem{cor}[thm]{Corollary}\crefname{cor}{corollary}{corollaries}
\newtheorem*{cor*}{Corollary}
\newtheorem{lem}[thm]{Lemma}\crefname{lem}{lemma}{lemmas}
\newtheorem*{lem*}{Lemma}
\newtheorem{prop}[thm]{Proposition}\crefname{prop}{proposition}{propositions}
\newtheorem*{prop*}{Proposition}
\crefname{note}{note}{notes}
\crefname{notation}{notation}{notation}
\crefname{notation}{notation}{notation}

\theoremstyle{remark}
\newtheorem{example}[thm]{Example}\crefname{example}{example}{examples}
\newtheorem*{example*}{Example}
\newtheorem{rem}[thm]{Remark}\crefname{rem}{remark}{remarks}
\newtheorem*{rem*}{Remark}

\newcommand{\ZZ}{\mathbb Z}
\newcommand{\QQ}{\mathbb Q}

\newcommand{\SO}{\mathrm{SO}}

\DeclareMathOperator{\diag}{diag}
\DeclareMathOperator{\id}{id}
\DeclareMathOperator{\rank}{rank}
\newcommand{\interrank}[2]{\mathrm{interrank}(#1,#2)}

\DeclareMathOperator{\im}{im}

\DeclareMathOperator{\Hom}{Hom}

\renewcommand{\H}{\mathrm{H}} 
\DeclareMathOperator{\Tor}{Tor}

\DeclareMathOperator{\projdim}{pd}

\DeclareMathOperator{\Spec}{Spec}
\newcommand{\support}[2][]{\mathrm{Supp}_{#1}#2}

\renewcommand{\iff}{\Leftrightarrow}

\newcommand{\sequence}[1]{\mathbf #1}
\newcommand{\ideal}[1]{\mathfrak #1}

\newcommand{\K}{\mathrm{K}}
\newcommand{\Rep}[1]{\mathrm{R}#1} 
\newcommand{\I}[1]{\ideal{I}#1} 
\newcommand{\Ip}[1]{\ideal{I}_p#1} 
\newcommand{\Weyl}{\mathrm W} 
\newcommand{\centre}[1]{\mathrm{Z} #1}
\newcommand{\centralizer}[2][]{\mathrm{Z}_{#1}#2}
\newcommand{\normalizer}[2][]{\mathrm{N}_{#1}#2}

\newcommand{\Grassmannian}[3][1]{\mathrm{Gr}_{#1}(#2,#3)} 

\newcommand{\Char}[1]{{\Hom(#1,S^1)}} 
\NewDocumentCommand{\koszul}{O{\bullet} m m}{\Lambda^{#1}_{#2}(#3)}

\newcommand{\ctext}[2]{\text{\parbox{#1}{\relax\ifvmode\centering\fi #2}}}

\newcommand{\ie}{i.\,e., }

\newcommand{\eg}{e.\,g.\ }

\newcommand{\back}{\backslash} 
\newcommand{\bibackslash}{\backslash\!\backslash} 
\newcommand{\biquotient}[3]{#2\back #1/#3}

\date{\today}
\title{An extension of Steinberg's Theorem to biquotient pairs of subgroups}
\author{Marcus Zibrowius\thanks{Research for this article was conducted within the framework of the DFG Research Training Group 2240:
Algebro-Geometric Methods in Algebra, Arithmetic and Topology.}
}
\begin{document}
\maketitle
\thispagestyle{empty}
\renewcommand{\abstractname}{}
\begin{abstract}
  We study the derived tensor product of the representation rings of subgroups of a given compact Lie group \(G\).
  That is, given two such subgroups \(H_1\) and \(H_2\), we study the tensor product of the associated representation rings \(\Rep H_1\) and \(\Rep H_2\) over the representation ring \(\Rep G\), and prove a vanishing result for the associated higher Tor-groups. This result can be viewed as a natural extension of the Theorem of Steinberg that asserts that the representation rings of maximal rank subgroups of \(G\) are free over \(\Rep G\).  It may also be viewed as an analogue of a result of Singhof on the cohomology of classifying spaces.  We include an immediate application to the complex K-theory of biquotient manifolds.
\end{abstract}

Consider a compact connected Lie group \(G\), with complex representation ring \(\Rep{G}\).  Given two closed connected subgroups \(H_1\) and \(H_2\), we can consider their representation rings \(\Rep{H_1}\) and \(\Rep{H_2}\) as \(\Rep{G}\)-modules.
The purpose of this paper is to prove:
\begin{thm}\label{thm:Tor-of-pairs-simplified}
  Suppose the fundamental group of \(G\) is torsion-free.  If \(H_1\) intersects every conjugate of \(H_2\) trivially, then
  \(
  \Tor_i^{\Rep G}(\Rep H_1, \Rep H_2)
  \)
  vanishes for all \(i > \rank G - (\rank H_1 + \rank H_2)\).
\end{thm}
This result is motivated by a theorem of Pittie and Steinberg, on which it builds and which it extends, and by a result of Singhof, of which it is an analogue.  We briefly discuss each motivation in turn, and explain the implications for the complex K-theory of biquotient manifolds.

The first motivation is the following result which, following \cite{Panin:twisted} and \cite{Ananyevskiy:homogeneous}, we simply refer to as \textit{Steinberg's Theorem}.\footnote{See \Cref{historical-remark} for a historical discussion.}
\begin{thm*}[Steinberg \cite{Steinberg:Pittie}]
  For any compact connected Lie group \(G\) with torsion-free fundamental group, and for any closed connected subgroup \(H\) of maximal rank (\ie \(\rank H = \rank G)\), the representation ring \(\Rep H\) is 
  free as an \(\Rep G\)-module.
\end{thm*}
Note that under the stated assumptions, a finitely generated \(\Rep G\)-module is free if and only if it is flat (see \Cref{flat-is-projective-is-free-for-regular-G}), so Steinberg's Theorem is equivalent to the assertion that \(\Tor_i^{\Rep G}(\Rep H,-)\) vanishes in all positive degrees~\(i\) for all connected subgroups \(H\) of maximal rank.
In particular -- and this is historically the case of most interest -- all higher Tor-groups \(\Tor_i^{\Rep G}(\Rep H,\ZZ)\) vanish for connected subgroups \(H\) of maximal rank, which immediately implies that a certain spectral sequence of Hodgkin computing the complex K-theory of the homogeneous space \(G/H\) collapses (see \Cref{sec:K} for details).  This case corresponds to the special case of \Cref{thm:Tor-of-pairs-simplified} when \(H_1\) is of maximal rank and \(H_2\) is trivial, and it is in this sense that \Cref{thm:Tor-of-pairs-simplified} is an extension.

\begin{@empty}
\newcommand{\B}{\mathrm{B}}

The motivating result of Singhof can be stated as follows:
\begin{thm*}[Singhof {\cite[Proposition~(6.4)]{Singhof:DCM}}]

  Let \(G\) be a compact connected Lie group (with arbitrary fundamental group), with closed connected subgroups \(H_1\) and \(H_2\).
  Consider the rational\footnote{
    The theorem holds more generally for cohomology with coefficients in any integral domain for which the three cohomology rings \(\H^*(\B G)\), \(\H^*(\B H_1)\) and \(\H^*(\B H_2)\) are polynomial algebras with generators in even dimensions.  This is always true over \(\QQ\); see \eg \cite[Theorem~6.38]{McCleary}.
  }
  cohomology rings \(\H^*(-) := \H^*(-;\QQ)\) of the classifying spaces \(\B G\), \(\B H_1\) and \(\B H_2\).
  If \(H_1\) intersects every conjugate of \(H_2\) trivially, then the graded \(\Tor\)-groups
  \(
  \Tor_i^{\H^*(\B G)}(\H^*(\B H_1), \H^*(B H_2))
  \)
  vanish for all \(i > \rank G - (\rank H_1 + \rank H_2)\).
\end{thm*}
\end{@empty}

We will refer to the condition on \(H_1\) and \(H_2\) appearing in \Cref{thm:Tor-of-pairs-simplified} and in the result of Singhof as the \textbf{strict biquotient condition}.
This condition arises rather naturally.  It is equivalent to the condition that \(H_1\) acts freely from the left on the homogeneous space \(G/H_2\) \cite[(1.1)]{Singhof:DCM}, and hence implies that the double coset space, or biquotient, \(\biquotient{G}{H_1}{H_2}\) is a smooth manifold.   We thus refer to \(\biquotient{G}{H_1}{H_2}\) as \textbf{biquotient manifold}.\footnote{
  In \cite{Singhof:DCM}, the strict biquotient condition is referred to as `double coset condition', and biquotient manifolds are called `double coset manifolds'. The term `biquotient' used here appears to be vastly more popular.
}

The study of such biquotient manifolds has a rich history.  We refer to \cite[Chapter~1]{DeVito:thesis} for a concise overview, in particular with respect to the quest for manifolds of positive curvature, and for an advertisement of biquotient manifolds as a fertile testing ground for conjectures relating the topology to the geometry of manifolds.  For classification results, see in particular  \cite{Eschenburg:thesis,Ziller:lectures,KapovitchZiller,Totaro:biquotients,DeVito:6and7}.

Singhof's main application of the above vanishing result was a computation of the cohomology of  biquotient manifolds, via an Eilenberg--Moore spectral sequence.
Similarly, in \Cref{sec:K}, we will use \Cref{thm:Tor-of-pairs-simplified} to compute the complex K-theory of biquotient manifolds, via the spectral sequence of Hodgkin already mentioned, generalizing the classical computation of the complex K-theory of homogeneous spaces:

\begin{cor}\label{K-of-biquotients}
  Consider a compact connected Lie group \(G\) with torsion-free fundamental group, closed connected subgroups \(H_1\) and \(H_2\) satisfying the strict biquotient condition, and the associated biquotient manifold \(\biquotient{G}{H_1}{H_2}\).
  If \(\rank H_1 + \rank H_2 \geq \rank G - 1\), we have the following isomorphisms:
  \begin{align*}
    \K^0(\biquotient{G}{H_1}{H_2}) &\cong \Rep H_1 \otimes_{\Rep G} \Rep H_2\\
    \K^1(\biquotient{G}{H_1}{H_2}) &\cong \Tor_1^{\Rep G}(\Rep H_1, \Rep H_2)
  \end{align*}
  In the maximal rank case, \ie when \(\rank H_1 + \rank H_2 = \rank G\), the groups in the second line vanish, \ie \(\K^1(\biquotient{G}{H_1}{H_2}) = 0\).
\end{cor}
Note that all concepts under consideration here depend only on the conjugacy classes of \(H_1\) and \(H_2\):
The representation rings of conjugate subgroups \(H_i\) and \(H_i'\) of \(G\) are isomorphic as \(\Rep{G}\)-modules (see \Cref{rem:Rep-of-conjugate-subgroups}), and the biquotients \(\biquotient{G}{H_1}{H_2}\) and \(\biquotient{G}{H_1'}{H_2'}\) are diffeomorphic.
\Cref{K-of-biquotients} implies that, under the stated assumptions, the natural map from \(\Rep(H_1\times H_2) \cong \Rep H_1\otimes_\ZZ \Rep H_2\) to \(\K^0(\biquotient{G}{H_1}{H_2})\) is surjective.  This provides a partial answer to a question raised in \cite{DeVitoGonzalez}.

\bigskip

We will establish \Cref{thm:Tor-of-pairs-simplified} in slightly greater generality than stated above.

First, the assumptions on \(G\) can be relaxed slightly.
Let's say a Lie group \(G\) is \textbf{good} if it is compact and connected and if \(\Rep G\) is a polynomial ring tensored with a Laurent ring.  All compact connected Lie groups with torsion-free fundamental group are good (\Cref{torsion-free-fundamental-group-implies-good}), but not conversely.
Steinberg established the theorem cited above under this more general assumption (see \Cref{thm:steinberg} below), and our \Cref{thm:Tor-of-pairs-simplified} likewise holds for all good~\(G\).

Second, the strict biquotient condition can be relaxed as follows.  Two subgroups \(H_1\) and \(H_2\) of \(G\) satisfy the \textbf{lax biquotient condition} if every intersection of \(H_1\) with a conjugate of \(H_2\) is central in \(G\).  In this case, the action of \(H_1\) on \(G/H_2\) is still effectively free, and the double coset space \(\biquotient{G}{H_1}{H_2}\) is still a smooth manifold.  \Cref{thm:Tor-of-pairs-simplified} also holds in this situation, but with an obvious correction term.  Given any pair of subgroups \(H_1\), \(H_2\), we define the \textbf{intersection rank} \(\interrank{H_1}{H_2}\) as the maximum of the ranks of the intersections \(H_1\cap gH_2 g^{-1}\) as \(g\) ranges over all elements of \(G\).  With this definition in place, the result reads as follows:

\begin{thm}\label{thm:Tor-of-pairs}
  Let \(G\) be a good Lie group, and let \(H_1\) and \(H_2\) be closed connected subgroups.
  If \(H_1\) and \(H_2\) satisfy the lax biquotient condition, then
  \[
  \Tor_i^{\Rep G}(\Rep H_1, \Rep H_2) = 0
  \]
  for all
  \(
    i > \rank G - (\rank H_1 + \rank H_2) + \interrank{H_1}{H_2}
  \).
\end{thm}

We say a few words about the proof.  Despite the close analogy of the statements, there is very little relation of our proof of \Cref{thm:Tor-of-pairs} with Singhof's proof of the result cited above, which makes crucial use of the grading in cohomology.

We first establish \Cref{thm:Tor-of-pairs} in the maximal rank case, \ie when \(\rank G = \rank H_1 + \rank H_2 - \interrank{H_1}{H_2}\).
When the ambient group \(G\) is a torus, this case can be verified by a direct calculation using the spectral sequence of a double complex.  The reader may easily reconstruct this calculation by stripping away all unnecessary complications from the proof of \Cref{maximal-rank-particular-subgroups} below.  It is also easy, for general \(G\), to reduce to the case that \(H_1\) and \(H_2\) are tori; see \Cref{red:torus}.  However, we have found no simple argument for passing from the statement for ambient tori to the statement for more general ambient groups~\(G\).  Instead, we try to salvage as much as possible from the actual calculation for tori, and generalize the whole argument.  One key ingredient is a “reduction to the diagonal”, which allows us to assume that \(\Rep{H_2}\) is a quotient of \(\Rep{G}\) by a regular sequence (see \Cref{diagonal-reformulation} and the proof of \Cref{maximal-rank-tori}).

The passage from the maximal rank case to the general case is achieved by enlarging \(H_1\) (see \Cref{enlarging-tori} and the proof of \Cref{arbitrary-subgroups}).
However, in general, this enlargment process destroys the biquotient condition (\Cref{enlargement-destroys-biquotient-condition}).
This forces us to prove \Cref{thm:Tor-of-pairs} in somewhat greater generality than stated above: we show that, for arbitrary pairs of connected subgroups, the relevant Tor groups vanish when localized at primes supporting the representation ring \(\Rep{(\centre{G})}\)  of the centre of \(G\) (see \Cref{arbitrary-subgroups}).
The key for extracting \Cref{thm:Tor-of-pairs} from this result is Segal's theory of supports of prime ideals in representation rings \cite{Segal:RG}, which allows us to translate the biquotient condition into commutative algebra (\Cref{supports}): the lax biquotient condition implies that the only primes of \(\Rep G\) supporting both \(\Rep H_1\) and \(\Rep H_2\) are those primes also supporting \(\Rep{(\centre{G})}\) (\Cref{eg:lax-biquotient-condition-translated-into-commutative-algebra}).  Thus, we may deduce that the relevant Tor groups vanish by arguing one prime \(\ideal p\) at a time:  if \(\ideal p\) supports \(\Rep{(\centre{G})}\), the Tor groups vanish at \(\ideal p\) by the more general result; if not, they vanish because at least one of \(\Rep{H_1}\) and \(\Rep{H_2}\) do so.

\paragraph{Acknowledgements}
I am very grateful to Wilhelm Singhof for asking me whether \Cref{thm:Tor-of-pairs-simplified} holds, and to David González-Álvaro and Panagiotis Konstantis for repeatedly motivating me to come up with an answer.  Jeff Carlson and Oliver Goertsches shared their thoughts on various preliminary approaches to the problem.
The special case of \Cref{thm:Tor-of-pairs-simplified} in which \(\rank H_1 = 1\) was established by Alessandra Wiechers in her Master's thesis \cite{Wiechers}, already using Segal's theory of supports.  Leopold Zoller outlined a strategy to me by which the first isomorphism of \Cref{K-of-biquotients} might alternatively be established in the case when \(G\) is a product of special unitary and symplectic groups, and \(H\) is of maximal rank, starting from the cohomological result of Singhof.
Jakob Bergqvist shared some helpful insights regarding \Cref{RG-biequidimensional-for-connected-G}.
Our use of the “reduction to the diagonal” is inspired by \cite[Chapter~V.C, \S\,4]{Serre:LA}.
Finally, I would like to thank the anonymous referee for careful reading and meticulous reports that have led to numerous improvements.

\vfill
\begin{@empty}
  \small
\renewcommand{\contentsname}{Overview}
\tableofcontents
\end{@empty}
\vfill
\newpage
\section{Representation rings}
We collect here assorted facts about complex representation rings of compact Lie groups that we will need later.

\subsection{Abelian groups}
For compact abelian Lie groups, the functor \(\Rep\) that takes a compact abelian group to its complex representation ring can be factored into a composition of two functors as follows:
\[
  \left(\ctext{6em}{compact abelian Lie~groups}\right) \xrightarrow[\Char{-}]{\simeq}
  \left(\ctext{6em}{finitely generated abelian groups}\right) \xrightarrow[{\ZZ[-]}]{}
  \left(\ctext{6em}{finitely generated commutative \(\ZZ\)-algebras}\right)
\]
The first functor, which takes a compact abelian Lie group \(S\) to its character group \(\Char{S}\), is a contravariant equivalence by Pontryagin duality.  The second functor takes an abelian group \(A\) to the associated group ring \(\ZZ[A]\).  In other words, \(\Rep{S}\cong \ZZ[\Char{S}]\) for any compact abelian Lie group~\(S\).  For example, when \(T\) is a torus of rank~\(r\), \ie when \(T\cong (S^1)^r\), we have \(\Hom(T,S^1) \cong \ZZ^r\).  So \(\Rep{T}\) is isomorphic to \(\ZZ[x_1^{\pm 1},\dots,x_r^{\pm 1}]\), a Laurent ring on \(r\) generators.
We will need the following consequences of this explicit description.

\begin{lem}\label{subgroups-of-tori}
  Consider a torus \(T\).
  \begin{compactenum}[(a)]
  \item For an arbitrary closed subgroup \(S\subseteq T\), the restriction \(\Rep{T}\to \Rep{S}\) is surjective.
    \label{subgroups-of-tori:restriction-surjective}
  \item Any subtorus \(S\subseteq T\) is a direct factor.  For such a subtorus \(S\), the kernel of the restriction \(\Rep{T}\to\Rep{S}\) is generated by a regular sequence of length \(\rank{T} -\rank{S}\).
    \label{subgroups-of-tori:restriction-to-subtori}
    \label{subgroups-of-tori:subtori-are-direct-factors}
  \item
    There is an inclusion-reversing bijection between subgroups \(S\) of \(T\) and subgroups \(K_S\) of \(\ZZ^r\), where \(r := \rank T\), defined as follows:
    Identify \(\Char{T}\) with \(\ZZ^r\).  Given a subgroup \(S\subseteq T\), define the associated subgroup of \(\ZZ^r\) as \( K_S := \ker(\Hom(T,S^1)\to \Hom(S,S^1))\).

    Under this correspondence, \(\rank{K_S} = r - \rank{S}\), the subtori of \(T\) correspond to the direct summands of \(\ZZ^r\), and \(K_{S_1\cap S_2} = K_{S_1} + K_{S_2}\).
    \label{subgroups-of-tori:bijection}

    \qed
  \end{compactenum}
\end{lem}

\subsection{Arbitrary compact groups}

\begin{defn}\label{def:rank}
  We define the \textbf{rank} of a compact Lie group \(G\) as the rank of any maximal torus contained in \(G\).
\end{defn}
Textbook sources such as \cite{BtD:Lie,Adams:LieLectures}) restrict this definition to connected groups, but there is no harm in extending it to disconnected groups.  Since any torus in \(G\) is necessarily contained in the identity component of \(G\), we are effectively defining the rank of \(G\) as the rank of this connected component.  In our main result, \(G\) and the subgroups \(H_i\) are assumed connected, but the intersections of \(H_1\) with (conjugates of) \(H_2\) may not be.

\begin{lem}\label{RG-for-general-G}
  For any compact Lie group \(G\):
  \begin{compactenum}[(a)]
  \item \(\Rep G\) is a free as an abelian group.
  \item \(\Rep G\) is finitely generated as a \(\ZZ\)-algebra, hence, in particular, a noetherian ring.
    \label{RG-for-general-G:noetherian}
  \item \(\dim \Rep G = \rank G + 1\)
    \label{RG-for-general-G:dimension}
  \end{compactenum}
\end{lem}
\begin{proof}
  The first observation follows directly from the definition of \(\Rep G\) and is only mentioned here because we will implicitly use it when applying~\Cref{Tor-CoR-II} later.
  For (\labelcref{RG-for-general-G:noetherian}), see \cite[Corollary~3.3]{Segal:RG}.

  For (\labelcref{RG-for-general-G:dimension}), consider first a compact abelian group \(S\).  Pullback along the projection to the identity component \(S\twoheadrightarrow S_0\) defines a finite ring extension \(\Rep{S_0}\hookrightarrow \Rep{S}\).  So \(\dim{\Rep{S}} = \dim{\Rep{S_0}}\). As \(\Rep{S_0} = \ZZ[x_1^{\pm 1}, \dots,x_r^{\pm 1}]\) with \(r = \rank S_0 = \rank S\), the claim follows in this case.

  For any compact connected \(G\), the restriction to a maximal torus \(T\) defines an integral extension \(\Rep{G}\to\Rep{T}\) (see \Cref{RG-for-connected-G} below). So the result for \(G\) follows from the result for \(T\) which we have just established.
  For general compact~\(G\), we similarly have an integral extension \(\Rep{G} \to \prod_S \Rep{S}\), where the product is over a finite number of compact abelian groups \(S\) \cite[proof of Proposition~(3.2)]{Segal:RG}.  So \(\dim\Rep{G} = \max_S \dim\Rep{S}\). More precisely, the groups \(S\) range over all conjugacy classes of what Segal calls ``Cartan subgroups'' of \(G\) (though note that this terminology conflicts with modern usage).  Not every such Cartan subgroup is of maximal rank, as the example in the final remark of \cite[\S\,1]{Segal:RG} shows.  However, any maximal torus of the identity component of \(G\) is a Cartan subgroup, so at least one of the groups \(S\) has maximal rank, and the claim follows.
\end{proof}

\begin{lem}\label{RH-finite-over-RG}
  For any subgroup \(H\) of a compact Lie group \(G\), \(\Rep H\) is finitely generated as an \(\Rep G\)-module.
\end{lem}
\begin{proof}
  This is \cite[Proposition~3.2]{Segal:RG}.
\end{proof}

\begin{rem}\label{rem:Rep-of-conjugate-subgroups}
  The representation ring \(\Rep{H'}\) of a subgroup \(H'\) conjugate to \(H\) is isomorphic to \(\Rep{H}\) as an \(\Rep G\)-module.
  This follows from the fact that conjugation by an inner automorphism induces the identity on \(\Rep{G}\).
\end{rem}

\subsection{Connected groups}\label{sec:connected-groups}
We now restrict our attention to compact \emph{connected}~\(G\).
Recall that all maximal tori of \(G\) are conjugate, and that any element of \(G\) is contained in a maximal torus \cite[Corollaries~4.22 \& 4.23]{Adams:LieLectures}.
Given a maximal torus \(T\subseteq G\), the \textbf{Weyl group} \(\Weyl = \Weyl(G,T)\) is defined as \(\Weyl := (\normalizer[G]{T})/T\).  This is always a finite group. It acts on \(\Rep{T}\) by conjugation, and we denote by \((\Rep{T})^\Weyl\) the subring of fixed points under this action.

\begin{lem}\label{RG-for-connected-G}
  Consider a compact connected Lie group \(G\) with maximal torus \(T\) and associated Weyl group \(\Weyl\).
  The restriction \(i^*\colon \Rep G\to \Rep T\) is a finite extension which restricts to an isomorphism  \(\Rep{G}\cong (\Rep{T})^\Weyl\).
  In particular:
  \begin{compactenum}[(a)]
  \item
    \(\Rep{G}\) is an integral domain.
    \label{RG-for-connected-G:integral-domain}
  \item The set of prime ideals of \(\Rep T\) lying over a fixed prime ideal \(\ideal p \subset \Rep G\) is equal to a single \(\Weyl\)-orbit of a prime in \(\Rep T\).
    \label{RG-for-connected-G:RT-primes}
  \item
    A prime ideal \(\ideal p_T\subset \Rep T\) lying over a given prime ideal \(\ideal p \subset \Rep G\) is unique with this property if and only if it is \(\Weyl\)-invariant.
    In this case, it contains every other prime ideal \(\ideal q \subset \Rep T\) with \((i^*)^{-1}\ideal q \subseteq \ideal p\), and
    \((\Rep T)_{\ideal p}\) is a local ring with maximal ideal \(\ideal p_T\cdot (\Rep T)_{\ideal p}\).
    \label{RG-for-connected-G:max-RT-primes}
  \end{compactenum}
\end{lem}
\begin{proof}
  For the identification of \(\Rep G\) with \((\Rep{T})^\Weyl\), see for example \cite[Theorem~6.20]{Adams:LieLectures}.
  For the finiteness of \(i^*\), see \Cref{RH-finite-over-RG} above.  Alternatively, use the identification of \(\Rep G\) with \((\Rep{T})^\Weyl\) and \cite{GoertzWedhorn:AGI}*{Proposition~12.27~(4)}.

  Claim~(\labelcref{RG-for-connected-G:integral-domain}) is immediate from the fact that \(\Rep{T}\) is an integral domain and the injectivity of~\(i^*\).

  For claim~(\labelcref{RG-for-connected-G:RT-primes}), see for example \cite{GoertzWedhorn:AGI}*{Proposition~12.27~(2) and (3)}. The claim spells out the general fact that, for any commutative unital ring \(A\), and any action of a finite group \(\Weyl\) through ring homomorphisms on \(A\), the prime spectrum \(\Spec(A^\Weyl)\) of the fixed ring can be naturally identified with \(\Spec(A)/\Weyl\), the \emph{set-theoretic} quotient of the prime spectrum of \(A\).  (The same statement is true in the category of affine schemes, in fact more obviously so -- see \cite{GoertzWedhorn:AGI}*{Proposition~12.27~(1)}. But it is the set-theoretic statement that we will need later.)

  The first part of claim~(\labelcref{RG-for-connected-G:max-RT-primes}) follows immediately from (\labelcref{RG-for-connected-G:RT-primes}).  The second part follows from the first and the going-up property, which is shared by all finite (or, more generally, integral) extensions. Finally, note that the maximal primes of \((\Rep T)_{\ideal p}\) correspond to the primes of \(\Rep T\) that lie over \(\ideal p\).
\end{proof}
\begin{example}\label{eg:max-ideals-of-RG}
  Let us write \(\I G\) for the augmentation ideal, \ie for the kernel of the rank homomorphism \(\rank\colon \Rep G \to \ZZ\).
  Clearly, \(\I T\) is a prime ideal of \(\Rep{T}\) lying over \(\I G\), and it is \(\Weyl\)-invariant.
  So \Cref{RG-for-connected-G} shows that \((\Rep T)_{\I G}\) is a local ring with maximal ideal \(\I T\cdot (\Rep T)_{\I G}\).

  Similarly, for an arbitrary prime \(p\in\ZZ\), we can consider the ideal \(\Ip G := \rank^{-1}((p))\).
  Again, \(\Ip T\) is a \(\Weyl\)-invariant prime ideal lying over \(\Ip G\), so \((\Rep T)_{\Ip G}\) is a local ring with maximal ideal \(\Ip T\cdot (\Rep T)_{\Ip G}\).
\end{example}

The inclusion of \(\Rep{G}\) into \(\Rep{T}\) moreover splits:
\begin{thm}[Holomorphic induction]\label{holomorphic-induction}
  For any compact connected Lie group \(G\) with maximal torus \(i\colon T \hookrightarrow G\), we have a morphism of \(\Rep G\)-modules \(i_*\colon \Rep T \to \Rep G\) such that \(i_*i^* = \id\).
\end{thm}
\begin{proof}
See  \cite{Atiyah:Bott}*{Remark~1 after Proposition~4.9} or \cite[\S\,4]{Segal:RG}, where \(i_*\) is written as \(i_c\) and the identity \(i_c i^* = \id\) follows from Proposition~4.1\,(iv).
\end{proof}

\begin{lem}\label{RG-biequidimensional-for-connected-G}
The representation ring \(\Rep{G}\) of any compact connected \(G\) is biequidimensional, in the strong sense that all saturated chains of prime ideals have equal length, and hence satisfies the dimension formula:  for any prime ideal \(\ideal p\subset \Rep G\), \[\dim \Rep G = \dim{\Rep G/\ideal p} + \dim{(\Rep G)_{\ideal p}}.\]
\end{lem}
\begin{proof}
  The dimensional formula holds in biquidimensional rings by \cite[Proposition~4.1]{Heinrich}.
  So we need only show that \(\Rep G\) is biequidimensional.
  We refer to \cite{Heinrich} for an in-depth discussion of this property.
  Suffice it to say that, in general, an equidimensional, coequidimensional and catenary ring is \emph{not} necessarily biequidimensional.
  However, it is readily verified that the following three properties do imply biequidimensionality of a ring \(R\):
  \begin{compactenum}[(i)]
  \item \(R\) has a unique minimal prime ideal,
  \item \(R\) is catenary, \ie all saturated chains of prime ideals between two given prime ideals in \(R\) have equal length, and
  \item \(R\) is equicodimensional, \ie all its maximal ideals have equal height.
  \end{compactenum}
  We now verify these two properties for \(R = \Rep G\).
  Property~(i) clearly holds as \(\Rep G\) is an integral domain (\Cref{RG-for-connected-G}~(\labelcref{RG-for-connected-G:integral-domain})).
  Property~(ii) holds for \(\Rep G\) as \(\ZZ\) is universally catenary, as is any Dedekind domain, and as \(\Rep{G}\) is finitely generated over \(\ZZ\) (\Cref{RG-for-general-G}~(\labelcref{RG-for-general-G:noetherian})).
  Property~(iii) holds for \(R = \Rep G\) in view of the following general fact:  any integral domain \(R\) finitely generated over a Noetherian Jacobson universally catenary and equicodimensional domain \(S\) is again equicodimensional \cite[Proposition~10.6.1 (i)]{EGA4}.  This applies in particular to \(S = \ZZ\), and more generally to any Dedekind domain \(S\) with infinitely many maximal ideals.
\end{proof}

We include here another lemma that we will use in conjunction with \Cref{RG-biequidimensional-for-connected-G}.
\begin{lem}\label{dimension-lemma}
  Suppose \(R\hookrightarrow S\) is a finite extension of noetherian rings, and that \(\ideal b \subseteq S\) is some ideal.
  For any prime ideal \(\ideal p\in\support[R]{(S/\ideal b)}\), there exists a prime ideal \(\ideal p'\in\support[R]{(S/\ideal b)}\) such that \(\ideal p'\subseteq \ideal p\) and such that
  \[
    \dim{S_{\ideal p}} - \dim{(S/\ideal b)_{\ideal p}} = \dim{R_{\ideal p}} - \dim{(R/\ideal p')_{\ideal p}}.
  \]
\end{lem}
\begin{proof}
  For the ideal \(\ideal a := R\cap \ideal b\), we have a finite extension \(R/\ideal a \hookrightarrow S/\ideal b\), and hence a finite extension \((R/\ideal a)_{\ideal p} \hookrightarrow (S/\ideal b)_{\ideal p}\) for any prime ideal \(\ideal p\subset R\).  In particular, \(\dim{(S/\ideal b)_{\ideal p}} = \dim{(R/\ideal a)_{\ideal p}}\) and   \(\dim{S_{\ideal p}} = \dim{R_{\ideal p}}\).
  Next, note that
  \(
    \support[R]{(S/\ideal b)} = \support[R]{(R/\ideal a)} = \{\ideal p \in \Spec(R) \mid \ideal a\subseteq \ideal p\}.
  \)
  Indeed, the first equality again follows from the observation that \((R/\ideal a)_{\ideal p} \hookrightarrow (S/\ideal b)_{\ideal p}\) is an extension, and the second equality is elementary. So the given prime ideal \(\ideal p \in\support[R]{(S/\ideal b)}\) contains \(\ideal a\), and
  \(
  \dim{(R/\ideal a)_{\ideal p}} = \dim{(R/\ideal p')_{\ideal p}}
  \)
  for some prime ideal \(\ideal p'\) of \(R\) minimal among the primes satisfying \(\ideal a \subseteq \ideal p' \subseteq \ideal p\).
  In particular, \(\ideal p' \in \support[R]{(S/\ideal b)}\).
\end{proof}

\subsection{Good groups and Steinberg's Theorem}
For any compact connected Lie group \(G\), the commutator subgroup \(G'\) is semisimple \cite[Theorem~5.22\,(b)]{Sepanski}.
Moreover, \(G\) can be identified with a quotient group \(G\cong (G'\times T)/\Gamma\), where \(T\) is a torus and \(\Gamma\) is a finite abelian subgroup of \(G'\times T\) intersecting \(G'\) trivially \cite[Theorem~5.22\,(a)]{Sepanski}.  Following Steinberg, we therefore refer to \(G'\) as the \textbf{semisimple component} of \(G\).
The following is the main result of \cite{Steinberg:Pittie}.

\begin{thm}[{\cite[Theorem~1.2]{Steinberg:Pittie}}]\label{thm:steinberg}
  For a compact connected Lie group \(G\), the following conditions are equivalent:
  \begin{compactenum}[(a)]
  \item
    The semisimple component \(G'\) of \(G\) is a direct product of simple groups, each simply-connected or isomorphic to \(\SO(2r+1)\) for some \(r\).
    \label{def:good:semisimple-component}
  \item
    The representation ring \(\Rep{G}\) is a tensor product of a polynomial ring and a Laurent ring, \ie isomorphic to the ring \(\ZZ[x_1,\dots,x_m,y_1^{\pm 1}, \dots,y_n^{\pm 1}]\) for certain \(m\) and \(n\).
    \label{def:good:representation-ring}
  \item
    For every connected subgroup \(H\) of \(G\) of maximal rank, \(\Rep{H}\) is free as an \(\Rep{G}\)-module.
    \label{def:good:subgroups}
  \end{compactenum}
\end{thm}
To avoid endlessly repeating these somewhat intricate conditions, we simply define a \textbf{good} Lie group to be a compact connected Lie group \(G\) that satisfies the equivalent conditions of \Cref{thm:steinberg}.  The result stated as \textit{Steinberg's Theorem} is just the implication (a~\(\Rightarrow\)~c) of this theorem combined with the following observation:

\begin{lem}\label{torsion-free-fundamental-group-implies-good}
  Any compact connected Lie group with torsion-free fundamental group is good.
\end{lem}
\begin{proof}
  Suppose \(G\) is compact connected with torsion-free fundamental group.  Writing \(G\) as \((G'\times T)/\Gamma\) as above, we find that the fundamental group of \(G'\) is also torsion-free.  As \(G'\) is semisimple, this implies that \(G'\) is simply-connected, and hence a direct product of simply-connected simple groups. So \(G\) satisfies condition~(\labelcref{def:good:semisimple-component}) of \Cref{thm:steinberg}.
\end{proof}

Condition~(b) in \Cref{thm:steinberg} implies:
\begin{lem}\label{RG-regular-for-good-G}
The representation ring \(\Rep{G}\) of any good group \(G\) is regular.
\end{lem}
\begin{proof}
  For good \(G\), \(\Rep G\) is a localization of a polynomial ring over the integers, which is clearly noetherian and of finite global dimension.
\end{proof}

Condition~(b) also implies that, for good \(G\), the notions of ``projective dimension'' and ``weak\slash flat dimension'' are equivalent for finitely generated \(\Rep G\)-modules. Indeed:
\begin{lem}\label{flat-is-projective-is-free-for-regular-G}
  For good \(G\), a finitely generated \(\Rep{G}\)-module is free if and only if is projective if and only it is flat.
\end{lem}
\begin{proof}
  The first two notions coincide by a generalization of the Quillen--Suslin theorem \citelist{\cite{Swan:Laurent}\cite{Swan:Gubeladze}} and \Cref{thm:steinberg}\,(b).  The second two coincide for finitely generated modules over arbitrary Noetherian rings.
\end{proof}

\begin{rem}[History]\label{historical-remark}
  The implication (a~\(\Rightarrow\)~b) of \Cref{thm:steinberg} is fairly elementary, and the implication (b~\(\Rightarrow\)~c), which is the main ingredient in the result stated as \textit{Steinberg's Theorem} in the introduction, was in fact proved slightly earlier by Pittie \cite{Pittie:Homogeneous}, modulo the above generalization of the Quillen--Suslin theorem.  Thus, from today's vantage point, the main contribution of Steinberg is a proof of the \emph{opposite} implication: a precise characterization of those compact Lie groups that have the desired property~(c).  Moreover, Steinberg's proof includes an explicit basis of \(\Rep H\) as an \(\Rep G\)-module.

We can easily sketch Pittie's short argument for the implication (b~\(\Rightarrow\)~c) based on the results and observations already included above: For a maximal torus \(T\) in a good group \(G\), both \(\Rep G\) and \(\Rep T\) are regular noetherian rings (\Cref{RG-regular-for-good-G}).  As \(\Rep T\) is in any case finitely generated as an \(\Rep G\)-module (\Cref{RH-finite-over-RG}), it thus follows from \cite[Chapter~III.D, \S\,5, Theorem~13\,(b)]{Serre:LA} that \(\Rep T\) is a projective \(\Rep G\)-module.  For a general subgroup of maximal rank \(H\), we may assume that \(H\) contains \(T\), and use holomorphic induction (\Cref{holomorphic-induction}) to deduce that \(\Rep H\) is a direct summand of \(\Rep T\).  In particular, \(\Rep H\) is also projective as an \(\Rep G\)-module.  Finally, as we have just seen in \Cref{flat-is-projective-is-free-for-regular-G}, this is today known to imply that \(\Rep H\) is free.
\end{rem}

While \Cref{thm:steinberg} only deals with subgroups of maximal rank, the following consequence for subgroups of arbitrary ranks was proved in \cite[Corollary~3.10]{SCSQ}:
\begin{cor}\label{cor:proj-dim}
  Let \(G\) be a good Lie group, and let \(H\subset G\) be an arbitrary closed connected subgroup.
  Then the projective dimension of \(\Rep H\) as an \(\Rep G\)-module is at most \(\rank G - \rank H\).
\end{cor}

We do not reproduce the short proof of \Cref{cor:proj-dim} here, but we note that besides holomorphic induction (\Cref{holomorphic-induction}) the main ingredient is a change-of-rings theorem for Tor.  We record here two such theorems that we will have occasion to use later.

\begin{thm}[Change of rings for Tor I]\label{Tor-CoR-I}
  Fix a morphism of commutative rings \(\Lambda\to \Gamma\), a \(\Lambda\)-module \(A\) and a \(\Gamma\)-module \(C\):
  \[
    \begin{tikzcd}[column sep=tiny,row sep=tiny]
      A \arrow[d,dash] & C \arrow[d,dash]\\
      \Lambda \ar[r] & \Gamma
    \end{tikzcd}
  \]
  If \(\Tor^\Lambda_*(A,\Gamma)\) vanishes in all positive degrees, then we have isomorphisms of \(\Lambda\)-modules
  \(
    \Tor^\Lambda_i(A,C) \cong \Tor^\Gamma_i(A\otimes_\Lambda \Gamma, C)
  \)
  in all degrees~\(i\).
\end{thm}
\begin{proof}
  See \cite[Chapter~VI, Proposition~4.1.1]{CartanEilenberg}.
\end{proof}

\begin{thm}[Change of rings for Tor II]\label{Tor-CoR-II}
  Consider commutative \(\ZZ\)-algebras \(\Gamma\), \(\Lambda\) and \(\Sigma\) which are free as abelian groups.
  Suppose \(A\) is a \(\Gamma\otimes\Lambda\)-module, \(B\) is a \(\Lambda\otimes\Sigma\)-module and \(C\) is a \(\Sigma\otimes\Gamma\)-module, as summarized in the following diagram:
  \[
    \begin{tikzcd}[column sep=tiny,row sep=tiny]
      A \arrow[dr,dash]\arrow[ddrr,dash,bend right] && \arrow[dl,dash] B \arrow[dr,dash] && C \arrow[dl,dash]\arrow[ddll,dash, bend left]\\
      & \Lambda && \Sigma \\
      &&  \Gamma
    \end{tikzcd}
  \]
  Suppose \(\Tor_*^\Lambda(A,B)\) and \(\Tor_*^\Sigma(B,C)\) vanish in all positive degrees.  Then we have natural isomorphisms of \(\Sigma\otimes\Gamma\)-modules
  \[
    \Tor_i^{\Gamma\otimes\Sigma}(A\otimes_{\Lambda}B,C) \cong \Tor_i^{\Lambda\otimes\Gamma}(A,B\otimes_{\Sigma}C)
  \]
  in all degrees~\(i\).
\end{thm}
\begin{proof}
  This is a special case of \cite[Chapter~IX, Theorem~2.8]{CartanEilenberg}.  Cartan and Eilenberg work with non-commutative \(K\)-projective algebras \(\Lambda\), \(\Gamma\), \(\Sigma\) over a commutative ring \(K\).  We have specialized the statement to the case of \(\ZZ\)-free commutative \(\ZZ\)-algebras. Note that the isomorphism \(r\) of \cite[Chapter~IX, Proposition~2.1]{CartanEilenberg}, on which the theorem relies, is an isomorphism of \(\Sigma\otimes\Gamma\)-algebras in this commutative situation.
\end{proof}

\subsection{Supports}
In \cite{Segal:RG}, Segal analyses the prime spectrum \(\Spec(\Rep{G})\) of an arbitrary compact Lie group \(G\).  In particular, Segal shows that, for each prime ideal \(\ideal p\subset \Rep{G}\), there is a subgroup \(S_{\ideal p}\), unique up to conjugation, which is minimal with respect to the property that \(\ideal p\) is the preimage of a prime ideal of \(\Rep(S_{\ideal p})\) under the restriction \(\Rep{G}\to \Rep{(S_{\ideal p})}\).  This subgroup \(S_{\ideal p}\), or rather any subgroup in this conjugacy class, is referred to as the \textbf{support} of~\(\ideal p\). The following \namecref{prop:segal} explains how Segal's ``supporting subgroups'' and the usual notion of ``prime ideals supporting a module'' are related.

\begin{prop}[Segal]\label{prop:segal}
  Let \(G\) be a compact Lie group, \(H\) a closed subgroup of \(G\)
  and \(\ideal p \in \Spec(\Rep G)\).
  The following statements are equivalent:
  \begin{compactenum}[(a)]
  \item \(\ideal p \in \im(\Spec(\Rep H) \to \Spec(\Rep G))\)
  \item \(\ideal p\) contains \(\ker(\Rep G \to \Rep H)\)
  \item \(\ideal p \in \support[\Rep G]{\Rep H}\)
  \item  The support \(S_{\ideal p}\) is conjugate to a subgroup of \(H\).
  \end{compactenum}
\end{prop}
\begin{proof}
  See \cite[Proposition~3.7\,(iv)]{Segal:RG} (or \cite[Proposition~110]{Wiechers}).
\end{proof}
\begin{prop}\label{supports}
  Let \(G\) be a compact connected Lie group with maximal torus~\(T\).
  \begin{compactenum}[(a)]
  \item \label{supports:subgroups}
    For subgroups \(H'\subseteq H \subseteq G\), we have
    \(
    \support[\Rep G]{\Rep H'} \subseteq \support[\Rep G]{\Rep H}.
    \)
  \item \label{supports:intersection}
    For any two subgroups \(H_1\) and \(H_2\),
    \begin{align*}
      \support[\Rep G]{(\Rep H_1 \otimes_{\Rep G} \Rep H_2)}
      &= \support[\Rep G]{\Rep H_1}\cap \support[\Rep G]{\Rep H_2}\\
      &= \bigcup_{g\in G}\support[\Rep G]\Rep(H_1\cap gH_2g^{-1}).
    \end{align*}
  \item \label{supports:dim}
    Given a subgroup \(H\subseteq G\) and a prime ideal \(\ideal p\in\support[\Rep G]{\Rep H}\),
    \[
      \dim{\Rep G/\ideal p}\leq \rank H + 1.
    \]
  \item
    Let \(\centre{G}\) denote the centre of \(G\).
    If \(\ideal p\in\support[\Rep G]{\Rep(\centre{G})}\), there is exactly one prime ideal of \(\Rep T\) lying over \(\ideal p\).
    \label{supports:centre}
  \end{compactenum}
\end{prop}
\begin{proof}
  Claim~(\labelcref{supports:subgroups}) is clear from the equivalence (c~\(\iff\)~d) of \Cref{prop:segal}.

  The first equality in Claim~(\labelcref{supports:intersection}) holds in general for finitely generated modules over a commutative ring.
  See \cite[Chapter~3, Exercise~19 (iv)]{AtiyahMacDonald} for the general fact,
  and \Cref{RH-finite-over-RG} for finite generation in the case at hand.
  The second equality in Claim~(\labelcref{supports:intersection}) again follows from the equivalence (c~\(\iff\)~d) of \Cref{prop:segal}.

  For claim~(\labelcref{supports:dim}), recall from \cite{Segal:RG} that the support \(S_{\ideal p}\) is topologically cyclic, \ie contains an element \(s\) whose powers are dense in \(S_{\ideal p}\). This element \(s\) is contained in some maximal torus of \(G\), so we may as well assume \(s\in T\) and hence \(S_{\ideal p}\subseteq T\).
  The ideal \(\ideal p\subset \Rep G\) is thus the inverse image of an ideal \(\ideal p_0\subset \Rep{(S_{\ideal p})}\) under a restriction map that factors as
  \[
    \Rep G \hookrightarrow \Rep T \twoheadrightarrow \Rep{(S_{\ideal p})}.
  \]
  Recall that any inclusion of a subgroup into a torus induces a surjection on representation rings (\Cref{subgroups-of-tori}~(\labelcref{subgroups-of-tori:restriction-surjective})).
  Writing \(\ideal p_T\subset \Rep T\) for the intermediate preimage of \(\ideal p_0\), we obtain induced morphisms of quotient rings as follows:
  \[
    \Rep G/\ideal p \hookrightarrow \Rep T/\ideal p_T \xrightarrow{\cong} \Rep{(S_{\ideal p})}/\ideal p_0
  \]
  Here, the first map is a finite extension (see \Cref{RG-for-connected-G}),
  so \(\dim{\Rep G/\ideal p} = \dim{\Rep T/\ideal p_T}\), and we deduce
  \(
  \dim{\Rep G/\ideal p} \leq \dim{\Rep{(S_{\ideal p})}}
  \).
  Moreover, by \Cref{RG-for-general-G}~(\labelcref{RG-for-general-G:dimension}), \(\dim{\Rep{(S_{\ideal p})}} = \rank{S_{\ideal p}} + 1\), and as \(S_{\ideal p}\) is conjugate to a subgroup of \(H\) we have \(\rank{S_{\ideal p}} \leq \rank H\).

  For claim~(\labelcref{supports:centre}), recall that the centre of \(G\) is contained in any maximal torus.  So the assumption implies that \(\ideal p\subseteq \Rep G\) can be obtained as the inverse image of some ideal in \(\Rep{(\centre G)}\) under a restriction map that factors as
  \[
    \Rep G \hookrightarrow \Rep T \twoheadrightarrow \Rep{(\centre G)}.
  \]
  Let us again write \(\ideal p_T\) for the intermediate preimage in \(\Rep T\), which by construction lies over \(\ideal p\).
  As the Weyl group \(\Weyl\) acts trivially on \(\centre G\), the restriction \(\Rep T \twoheadrightarrow \Rep{(\centre G)}\) is equivariant with respect to the canonical \(\Weyl\)-action on \(\Rep T\) and the trivial action on \(\Rep{(\centre G)}\). So \(\ideal p_T\subset \Rep T\) is \(\Weyl\)-invariant, and it follows from \Cref{RG-for-connected-G}~(\labelcref{RG-for-connected-G:RT-primes}) that \(\ideal p_T\) is the only ideal over \(\ideal p\).
\end{proof}

For the biquotient conditions appearing in the following two examples, see the introduction or \Cref{def:interrank} immediately below.
\begin{example}[Strict biquotient condition]\label{eg:strict-biquotient-condition-translated-into-commutative-algebra}
  If \((G,H_1,H_2)\) satisfy the strict biquotient condition, then the only prime ideals \(\ideal p\) of \(\Rep G\) supporting both \(\Rep H_1\) and \(\Rep H_2\)
  are the prime ideals \(\I G\) and \(\Ip G\) from \Cref{eg:max-ideals-of-RG}.
  Indeed, \Cref{supports}~(\labelcref{supports:intersection}) identifies the intersection \(\support[\Rep G]{\Rep H_1}\cap \support[\Rep G]{\Rep H_2}\) with \(\support[\Rep G]{\Rep(1)}\), and by the equivalence (b~\(\iff\)~c) of \Cref{prop:segal} this latter support consists precisely of the primes containing \(\I G\).
\end{example}
\begin{example}[Lax biquotient condition]\label{eg:lax-biquotient-condition-translated-into-commutative-algebra}
  If \((G,H_1,H_2)\) satisfy the lax biquotient condition, then any prime ideal \(\ideal p\) of \(\Rep G\) supporting both \(\Rep H_1\) and \(\Rep H_2\) also supports \({\Rep{(\centre G)}}\).
  Indeed, by \Cref{supports}~(\labelcref{supports:intersection}), any such prime \(\ideal p\) lies in \(\support[\Rep G]{\Rep{(H_1 \cap gH_2g^{-1})}}\) for some \(g\in G\), each of the intersections \(H_1 \cap gH_2g^{-1}\) is contained in the centre \(\centre{G}\) by assumption, and hence \(\ideal p\) lies in \(\support[\Rep G]{\Rep{(\centre G)}}\) by \Cref{supports}~(\labelcref{supports:subgroups}).
\end{example}

\section{Biquotient pairs}
In this section, we gather preliminary results on biquotients that facilitate certain reduction steps in the proof of \Cref{thm:Tor-of-pairs}.  We begin by recalling the key definitions from the introduction.

\begin{defn}\label{def:interrank}
  A triple \((G,H_1,H_2)\) consisting of a compact Lie group \(G\) and two closed subgroups \(H_1\) and \(H_2\) of \(G\) satisfies the \textbf{strict biquotient condition} if \(H_1\cap gH_2g^{-1} = \{1\}\) for every \(g\in G\).  It satisfies the \textbf{lax biquotient condition} if \(H_1\cap gH_2g^{-1}\) is central in \(G\) for every \(g\in G\). The \textbf{intersection rank} of \(H_1\) and \(H_2\) in \(G\) is defined as
  \[
    \interrank{H_1}{H_2} := \max_{g\in G} (\rank{H_1 \cap gH_2g^{-1}}).
  \]
  A \textbf{biquotient manifold} is a double coset space \(\biquotient{G}{H_1}{H_2}\) associated with a triple satisfying the (strict or) lax biquotient condition, with its induced smooth structure.
\end{defn}

\subsection{Reduction to diagonal}

A biquotient manifold \(\biquotient{G}{H_1}{H_2}\) can equivalently be viewed as a quotient
\[
  (H_1\times H_2)\bibackslash G := \biquotient{(G\times G)}{(H_1\times H_2)}{\Delta G},
\]
where \(\Delta G\) denotes the diagonal subgroup of \(G\times G\). Explicitly, the isomorphism is given by \(g \mapsto (g,1)\), with inverse \(g_1g_2^{-1}\mapsfrom
(g_1,g_2)\).

\begin{prop}\label{diagonal-reformulation}
  Let \(G\) be a compact Lie group with closed subgroups \(H_1\) and~\(H_2\).
  Denote as before by \(\centre{G}\) the centre of \(G\), and by \(\Delta\colon G \hookrightarrow G\times G\) the inclusion of the diagonal subgroup.
  \begin{compactenum}[(a)]
  \item
    \label{diagonal-reformulation:biquotient-condition}
    The triple \((G,H_1,H_2)\) satisfies the lax\slash strict biquotient condition if and only if the triple \((G\times G,H_1\times H_2, \Delta G)\) does.
  \item
    \label{diagonal-reformulation:interrank}
    The intersection ranks satisfy the equality
    \[\interrank{H_1}{H_2} = \interrank{H_1\times H_2}{\Delta G}.\]%
  \item
    \label{diagonal-reformulation:Tor}
    Considering \(\Rep G\) as an \(\Rep{(G\times G)}\)-module via restriction along \(\Delta\),
    we have canonical isomorphisms of abelian groups 
    \[
      \Tor_i^{\Rep(G\times G)}\left(\Rep(H_1\times H_2),\Rep G\right) \cong \Tor_i^{\Rep G}(\Rep H_1,\Rep H_2)
    \]
    in all degrees~\(i\). Likewise, for any prime ideal \(\ideal p \subset \Rep G\), we have isomorphisms of abelian groups
    \[
      \Tor_i^{\Rep(G\times G)}(\Rep(H_1\times H_2),\Rep G)_{(\Delta^*)^{-1}(\ideal p)} \cong \Tor_i^{\Rep G}(\Rep H_1,\Rep H_2)_{\ideal p}
    \]
    in all degrees~\(i\).
  \item
    \label{diagonal-reformulation:support}
    For a prime ideal \(\ideal p \subset \Rep G\),
    \begin{alignat*}{13}
      &\text{if } &\ideal p&\in\support[\Rep G](\Rep{(\centre G)}),\\
      &\text{then }& (\Delta^*)^{-1}(\ideal p) &\in \support[\Rep(G\times G)]{\Rep{(\centre{(G\times G)})}}.
    \end{alignat*}
  \end{compactenum}
\end{prop}
\begin{proof}
  For claims~(\labelcref{diagonal-reformulation:biquotient-condition}) and (\labelcref{diagonal-reformulation:interrank}), note that \((H_1\times H_2) \cap (g_1,g_2)\Delta G(g_1,g_2)^{-1}\) is isomorphic to \(H_1\cap g_1g_2^{-1} H_2 g_2g_1^{-1} \) under projection to the first coordinate, and that \(\centre{(G\times G)} = \centre G \times \centre G\).

  \newcommand{\catmod}[1]{#1\mathrm{-mod}}
  \newcommand{\GxG}{G^2}

  The first isomorphism of Tor-modules in claim~(\labelcref{diagonal-reformulation:Tor}) can be obtained from \Cref{Tor-CoR-II} by taking \(\Gamma\), \(\Lambda\), \(\Sigma\), \(A\), \(B\) and \(C\) as indicated in the following diagram:
  \[
    \begin{tikzcd}[column sep=tiny,row sep=tiny]
      \Rep{H_1} \arrow[dr,dash]\arrow[ddrr,dash,bend right] && \arrow[dl,dash] \Rep{H_2} \arrow[dr,dash] && \Rep{G} \arrow[dl,equal]\arrow[ddll,equal, bend left]\\
      & \ZZ && \Rep{G} \\
      &&  \Rep{G}
    \end{tikzcd}
  \]
  The Tor-modules assumed to vanish in positive degrees in \Cref{Tor-CoR-II} are \(\Tor_*^\ZZ(\Rep{H_1},\Rep{H_2})\) and \(\Tor_*^{\Rep{G}}(\Rep{H_2},\Rep{G})\).
  In both cases, this vanishing is clear because at least one of the modules is free over the considered base ring. So we obtain isomorphisms of \(\Rep{G}\otimes\Rep{G}\)-modules
  \[
    \Tor_i^{\Rep{G}\otimes\Rep{G}}(\Rep{H_1}\otimes\Rep{H_2},\Rep{G}) \cong \Tor_i^{\Rep{G}}(\Rep{H_1},\Rep{H_2})
  \]
  in all degrees~\(i\).
  To simplify notation, let us temporarily write \(\GxG\) for \(G\times G\).
  Recall that we have natural isomorphisms \(\Rep{(H_1\times H_2)}\cong \Rep{H_1}\otimes\Rep{H_2}\) and, in particular, \(\Rep(\GxG) \cong \Rep G \otimes \Rep G\).
  Under the latter isomorphism, the ring multiplication \(\Rep G \otimes \Rep G\to \Rep G\) gets identified with the restriction \(\Delta^*\colon \Rep{(\GxG)} \to \Rep G\).
  We can therefore rewrite the above isomorphisms in the form asserted in the \namecref{diagonal-reformulation}.
  The second, localized claim follows immediately.

  For claim~(\labelcref{diagonal-reformulation:support}), recall again that \(\centre{(\GxG)} = \centre{G}\times \centre{G}\).
  The surjection \(\Delta^*\colon \Rep(\GxG)\twoheadrightarrow \Rep G\) restricts to a surjection \(\Rep(\centre{(\GxG)})\to \Rep(\centre G)\), which we can view as a surjection of \(\Rep(\GxG)\)-modules.
  Localizing at \((\Delta^*)^{-1}(\ideal p)\) for some prime ideal \(\ideal p\) of \(\Rep G\) yields a surjection \(\Rep(\centre{(\GxG)})_{(\Delta^*)^{-1}(\ideal p)} \to (\Rep(\centre G))_{\ideal p}\).
  So if \(\Rep(\centre G)\) is supported at \(\ideal p\), then \(\Rep(\centre{(\GxG)})\) must be supported at \((\Delta^*)^{-1}(\ideal p)\).
\end{proof}

\subsection{The biquotient condition can be checked on finitely many intersections}
By definition, the biquotient condition is a condition on infinitely many intersections \(H_1\cap gH_2g^{-1}\). We show here that only finitely many intersections need to be taken into account, and that, similarly, the intersection rank of \(H_1\) and \(H_2\) can be computed from only finitely many intersections.  Both arguments proceed as follows:  we can replace \(H_1\) and \(H_2\) by tori (\Cref{interrank-can-be-computed-on-Tori}), and for tori we need only consider conjugates by elements of the Weyl group of~\(G\) (\Cref{biquotient-condition-of-tori-can-be-checked-from-Weyl-conjugates}).

\begin{prop}\label{interrank-can-be-computed-on-Tori}\label{biquotient-condition-can-be-checked-on-tori}
  Consider a compact Lie group \(G\).
  For any pair of closed connected subgroups \(H_1\) and \(H_2\), with maximal tori \(T_1\) and \(T_2\), respectively,
  \(\interrank{H_1}{H_2} = \interrank{T_1}{T_2}\).
  Moreover, the triple \((G,H_1,H_2)\) satisfies the lax\slash strict biquotient condition if and only if the triple \((G,T_1,T_2)\) does.
\end{prop}
\begin{proof}
  By definition, \(\interrank{H_1}{H_2}\) is the maximum over the ranks of all tori \(S\) contained in the intersections \(H_1\cap gH_2g^{-1}\), where \(g\) ranges over the elements of \(G\).  Any such torus \(S\) is conjugate to a subgroup of \(T_1\) and a subgroup of \(T_2\), hence conjugate to a subgroup of \(T_1\cap g'T_2g'^{-1}\) for some \(g'\in G\).
  This shows that the interranks agree.
   If \(H_1\) and \(H_2\) satisfy the lax or strict biquotient condition, then clearly so do \(T_1\) and \(T_2\).  Conversely, suppose \(T_1\) and \(T_2\) satisfy the lax (or strict) biquotient condition, and suppose that \(h\in H_1\cap gH_2g^{-1}\) for some \(g\in G\).  As any element of \(H_1\) is conjugate to an element of \(T_1\), and any element of \(H_2\) is conjugate to an element of \(T_2\), we find that \(h = h_1t_1h_1^{-1} = gh_2t_2h_2^{-1}g^{-1}\) for certain \(h_i\in H_i\) and \(t_i\in T_i\).  It follows that \(t_1 \in T_1 \cap g'T_2g'^{-1}\) for \(g' := h_1^{-1}gh_2\), hence \(t_1\) is central (or trivial) by our assumption on \(T_1\) and \(T_2\).  So \(h = h_1t_1h_1^{-1} = t_1\) is central (or trivial, respectively).
\end{proof}

Suppose now that \(G\) is compact and \emph{connected}, with a chosen maximal torus \(T\) and associated Weyl group \(\Weyl = (\normalizer[G]{T})/T\) as in \cref{sec:connected-groups}. Recall that \(\Weyl\) is finite.
Given a subgroup \(S\subseteq T\), we write \(w(S)\) for the conjugate \(hSh^{-1}\) given by an arbitrary representative \(h \in \normalizer[G]{T}\) of \(w\in \Weyl\).

\begin{lem}\label{reducing-conjugates-to-Weyl-conjugates}
  Consider a compact connected Lie group \(G\) with a maximal torus \(T\) and associated Weyl group \(\Weyl\).
  Suppose \(T_1\) and \(T_2\) are subtori of \(T\).
  For each \(g\in G\), there exists a \(w\in\Weyl\) such that \(T_1\cap gT_2g^{-1} \subseteq T_1 \cap w(T_2)\).
\end{lem}
\begin{proof}
  It suffices to show that, for every element \(g\in G\), there exists an element \(w\in\Weyl\) such that \(T\cap gT_2g^{-1} \subseteq w(T_2)\).
  We show this by adapting arguments from the proof of \cite{BtD:Lie}*{IV, Lemma~2.5}. Consider \(S:= T\cap gT_2g^{-1}\).  Conjugation by \(g\) restricts to an isomorphism between centralizers:
  \(
    c(g)\colon \centralizer[G]{(g^{-1}Sg)} \to \centralizer[G]{S}
  \).
  As \(g^{-1}Sg\subseteq T\), and as \(T\) is abelian, \(T\subseteq \centralizer[G]{(g^{-1}Sg)}\), so \(c(g)(T)\subseteq\centralizer[G]{S}\).
  Likewise, \(T\subseteq\centralizer[G]{S}\).
  So \(c(g)(T)\) and \(T\) are two maximal tori contained in (the identity component of) \(\centralizer[G]{S}\), and hence conjugate in \(\centralizer[G]{S}\).  So there exists some element \(f\in \centralizer[G]{S}\) such that \(fgTg^{-1}f^{-1} = T\), \ie such that \(fg\in \normalizer[G]{T}\).  Take \(w\) to be the image of \(fg\) in \(\Weyl\).  The equality \(S = fSf^{-1}\) now expands to
  \(
  T\cap gT_2g^{-1} = f(T\cap gT_2g^{-1})f^{-1}
  \),
  and the group on the right is clearly contained in \(fg T_2 g^{-1}f^{-1} = w(T_2)\).
\end{proof}

The following \namecref{interrank-of-tori-can-be-computed-from-Weyl-conjugates} is immediate from \Cref{reducing-conjugates-to-Weyl-conjugates}.
\begin{prop}\label{interrank-of-tori-can-be-computed-from-Weyl-conjugates}\label{biquotient-condition-of-tori-can-be-checked-from-Weyl-conjugates}
  Consider a compact connected Lie group \(G\) with a maximal torus \(T\) and associated Weyl group \(\Weyl\).
  For any pair of subtori \(T_1\) and \(T_2\) of \(T\),
  \(
  \interrank{T_1}{T_2} = \max_{w\in\Weyl}(\rank{T_1\cap w(T_2)})
  \).
  Moreover, the triple \((G,T_1,T_2)\) satisfies the lax/strict biquotient condition if and only if the intersection \(T_1\cap w(T_2)\) is trivial/central for each \(w\in\Weyl\).
  \qed
\end{prop}
Note that we can always arrange for the maximal tori \(T_1\) and \(T_2\) of two closed connected subgroups \(H_1\) and \(H_2\) to be contained in \(T\) by passing to conjugates of \(H_1\) and \(H_2\) if necessary.

\subsection{Enlarging tori while keeping intersections almost intact}

The ranks of two closed subgroups \(H_1\) and \(H_2\) of a compact Lie group \(G\) are clearly bounded by the rank of \(G\) in the sense that
\[
  \rank{G} \geq \rank{H_1} + \rank {H_2} - \interrank{H_1}{H_2}.
\]
The following \namecref{enlarging-tori} shows that, when \(H_1\) and \(H_2\) are tori, we can always enlarge \(H_1\) in such a way that the above inequality becomes an equality, but without affecting the intersection rank.

\begin{prop}\label{enlarging-tori}
  Consider a compact connected Lie group \(G\).
  Given two tori \(T_1,T_2\subseteq G\), there exists a torus \(T_1^+\) containing \(T_1\) such that
  \begin{flalign*}
    &\rank G = \rank T_1^+ + \rank T_2 - \interrank{T_1^+}{T_2}  \quad \text{ and}\\
    &          \interrank{T_1^+}{T_2} = \interrank{T_1}{T_2}.
  \end{flalign*}
\end{prop}
\begin{rem}\label{side-rem:fixed-maximal-torus}
  The proof will show that, given a maximal torus \(T\) of \(G\) containing \(T_1\),  we can choose \(T_1^+\) to be intermediate between \(T_1\) and \(T\), \ie such that \(T_1 \subseteq T_1^+ \subseteq T\).
\end{rem}
\begin{rem}\label{enlargement-destroys-biquotient-condition}
  Even when the triple \((G,T_1,T_2)\) satisfies the strict biquotient condition, it is not in general possible to find a torus \(T_1^+\) satisfying the above conditions and such that \((G,T_1^+,T_2)\) still satisfies one of the biquotient conditions.  In passing from \(T_1\) to \(T_1^+\), it can happen that the intersections \(T_1^+\cap w(T_2)\) acquire additional, non-central points.  For example, consider the rank-two group \(G := \mathrm{SU}(3)\) with its standard diagonal torus \(T = \{\diag (a,b,a^{-1}b^{-1}) \}\), the trivial subtorus \(T_1 := \{\diag (1,1,1)\}\) and the rank-one subtorus \(T_2 := \{\diag (a,a^{-1},1)\} \) defined by one of the simple coroots.  Clearly, the triple \((G,T_1,T_2)\) satisfies the strict biquotient condition.  However, any rank-one subtorus \(T_1^+ \subseteq T\) intersects at least one of the conjugates \(T_2\), \(T_2' := \{\diag (1,a,a^{-1})\}\) and \(T_2'':= \{\diag (a, 1, a^{-1})\}\) of \(T_2\) non-trivially, in a non-central point, so that \((G,T_1^+,T_2)\) does not even satisfy the lax biquotient condition.
\end{rem}
\begin{proof}[Proof of \Cref{enlarging-tori}]
  \newenvironment{tagged}[1]{%
    \par%
    \smallskip
    \begin{tabular}{m{1.5em}@{\hspace{1em}}m{0.85\linewidth}}%
      \centering (#1) & %
                        }{%
    \end{tabular}%
    \par%
    \smallskip
  }
  Choose a maximal torus \(T\) of \(G\) such that \(T_1 \subseteq T\). Replacing \(T_2\) by a conjugate subgroup if necessary, we may assume that also \(T_2\subseteq T\).
  Then by \Cref{interrank-of-tori-can-be-computed-from-Weyl-conjugates},
  \[ \interrank{T_1}{T_2} = \max_{w\in\Weyl}\rank{(T_1\cap w(T_2))}.\]
  Let us abbreviate the ranks as follows: \(r := \rank T\), \(r_1 := \rank T_1\), \(r_2 := \rank T_2\), \(r' := \interrank{T_1}{T_2}\).
  In view of \Cref{reducing-conjugates-to-Weyl-conjugates}, it suffices to construct a torus \(T_1^+\) intermediate between \(T_1\) and \(T\) and of rank \(r_1^+ := r - r_2 + r'\) such that
  \(
    \rank{(T_1^+\cap w(T_2))} = r'
  \)
  for each of the finitely many Weyl group elements \(w\in \Weyl\).

  The proof works by translating the statement from subgroups of \(T\) to subgroups of \(\ZZ^r\) to vector subspaces of \(\QQ^r\), where it can be solved essentially by linear algebra. The following paragraphs provide some details.

  First, let us abstract away the Weyl group and formulate our claim more generally as follows.
    \begin{tagged}{T}
    Suppose \(T_1\) and \(T_2^{(1)},\dots,T_2^{(n)}\) are subtori of a torus \(T\), with all \(T_2^{(i)}\) of equal rank.  Let \(r := \rank T\), \(r_1 := \rank T_1\), \(r_2 := \rank T_2^{(i)}\), \(r' := \max_i\rank(T_1 \cap T_2^{(i)})\) and \(r_1^+ := r-r_2+r'\).  There exists a torus \(T_1^+\) such that \(T_1\subseteq T_1^+ \subseteq T\), \(r_1^+ = \rank T_1^+\) and, for each \(i\), \(r' = \rank(T_1^+ \cap T_2^{(i)})\).
  \end{tagged}
  Using \Cref{subgroups-of-tori}, (T) may be translated into the following equivalent statement:
  \begin{tagged}{Z}
    Suppose \(K_1\) and \(K_2^{(1)},\dots,K_2^{(n)}\) are direct summands of \(\ZZ^r\), with all \(K_2^{(i)}\) of equal rank.  Let \(c_1 := \rank K_1\), \(c_2 := \rank K_2^{(i)}\), \(c' := \min_i\rank(K_1 + K_2^{(i)})\) and \(c_1^- := c' - c_2\).  There exists a direct summand \(K_1^-\) of \(\ZZ^r\) such that  \(K_1^- \subseteq K_1\) with \(c_1^- = \rank K_1^-\) and, for each \(i\), \(c' = \rank(K_1^- + K_2^{(i)})\).
  \end{tagged}
  We claim that it suffices to prove the following analogous statement (Q) about rational vector spaces:
  \begin{tagged}{Q}
    Suppose \(W_1\) and \(W_2^{(1)},\dots,W_2^{(n)}\) are vector subspaces of \(\QQ^r\), with all \(W_2^{(i)}\) of equal dimension.   Let \(c_1 := \dim W_1\), \(c_2 := \dim W_2^{(i)}\), \(c' := \min_i\dim(W_1 + W_2^{(i)})\) and \(c_1^- := c' - c_2\).  There exists a vector subspace \(W_1^- \subseteq W_1\) such that \(c_1^- = \dim W_1^-\) and, for each \(i\), \(c' = \dim(W_1^- + W_2^{(i)})\).
  \end{tagged}
  To deduce (Z) from (Q), note that there is a close relation between vector subspaces of \(\QQ^r\) and subgroups of \(\ZZ^r\).  For a vector subspace \(W\subseteq \QQ^r\), the intersection \(W\cap \ZZ^r\) is a direct summand of \(\ZZ^r\), of rank equal to the dimension of~\(W\). Conversely, from \emph{any} subgroup \(K\subseteq \ZZ^r\), we can obtain the vector subspace \(K\otimes_{\ZZ}\QQ \subseteq \QQ^r\) of dimension equal to the rank of \(K\). We obtain in this way a bijection between vector subspaces of \(\QQ^r\) and direct summands of \(\ZZ^r\).  Moreover, for arbitrary subgroups \(K_1, K_2 \subseteq \ZZ^r\), we have \((K_1 + K_2)\otimes\QQ = K_1\otimes\QQ + K_2 \otimes \QQ\).  (Note that \(K_1 + K_2\) may not be a direct summand even when \(K_1\) and \(K_2\) are.)
Thus, given direct summands \(K_1\) and \(K_2^{(i)}\) as in (Z), we may consider the vector spaces \(W_1 := K_1\otimes\QQ\) and \(W_2^{(i)} := K_2^{(i)}\otimes\QQ\) in (Q), obtain a vector space \(W_1^-\subseteq W_1\), and then take \(K_1^- := W_1^-\cap \ZZ^r\).  As noted, this is a direct summand of \(\ZZ^r\), clearly contained in \(K_1\), and the ranks work out as intended.

Next, we reduce (Q) to the case \(n=1\):
\begin{tagged}{Q$^\circ$}
  Let \(W_1\) and \(W_2\) be vector subspaces of \(\QQ^r\), with \(c_i := \dim W_i\).
  For any integer \(c_1^-\) such that \(0\leq c_1^-\leq \dim(W_1 + W_2) - c_2\),
  there exists a vector subspace \(W_1^-\subseteq W_1\) of dimension \(c_1^-\) such that
  \(\dim(W_1^- + W_2) = c_1^- + c_2\).
  Moreover, the vector subspaces \(W_1^-\) with these properties form a (non-empty) Zariski-open subset of the rational Grassmannian \(\Grassmannian[\QQ]{c_1^-}{W_1}\).
\end{tagged}
  To deduce (Q) from (Q$^\circ$), note that for each \(i\) we have \(c_2 \leq c' \leq \dim(W_1 + W_2^{(i)})\) in (Q), so \(0 \leq c_1^- \leq \dim(W_1 + W_2^{(i)}) - c_2\).  Thus, by (Q$^\circ$), for each \(i\) the vector subspaces \(W_1^-\subseteq W_1\) with \(\dim(W_1^- + W_2^{(i)}) = c'\) form a non-empty Zariski-open subset of \(\Grassmannian[\QQ]{c_1^-}{W_1}\).  So we can pick any \(W_1^-\) that lies in the intersection of these finitely many Zariski open subsets.

  Finally, (Q$^\circ$) can be phrased more elegantly, and proved, by noting that \(\dim(W_1+W_2) - c_2 = c_1 - \dim(W_1 \cap W_2)\), and by noting that the condition on the dimension of \(W_1^-+W_2\) is equivalent to the condition \(W_1^- \cap W_2 = \{0\}\).
  Such a \(W_1^-\) exists because we can simply choose a complement \(W'\) of \(W_1\cap W_2\) in \(W_1\), and then pick an arbitrary subspace \(W_1^- \subseteq W'\) of dimension \(c_1^-\).  To see that trivial intersection with \(W_2\) is an open condition on \(\Grassmannian[\QQ]{c_1^-}{W_1}\), note that this condition on \(W_1^-\) can be rephrased as follows:  the restriction of the natural projection \(W_1\twoheadrightarrow W_1/(W_1\cap W_2)\) to \(W_1^-\) is a monomorphism.  By change of coordinates, we may assume without loss of generality that this projection is the projection onto the first \(\dim(W_1+W_2) - c_2\) coordinates.  Then the restriction is a monomorphism if and only if, for a given basis \((b_1,\dots,b_{c_1^-})\) of \(W_1^-\), one of the maximal minors is non-zero.
  These \((c_1^-\times c_1^-)\)-minors are precisely the Plücker coordinates on the Grassmannian, so this is an open condition.
\end{proof}

\subsection{Reduction to toroidal subgroups}

\begin{prop}[tori vs.\ general case]\label{red:torus}
  Consider a compact Lie group \(G\) with connected subgroups \(H_1\) and \(H_2\), and suppose \(T_1\) and \(T_2\) are maximal tori of \(H_1\) and \(H_2\), respectively.
  If
  \(\Tor_i^{\Rep G}(\Rep T_1, \Rep T_2)_{\ideal p}\) vanishes for some degree \(i\) and some prime ideal \(\ideal p\subset \Rep G\), then so does
  \(\Tor_i^{\Rep G}(\Rep H_1, \Rep H_2)_{\ideal p}\).
\end{prop}
In case \(H_1\) and \(H_2\) are good, \Cref{thm:steinberg} can be used to show that the implication of the \namecref{red:torus} is even an equivalence, but we will not need this.
\begin{proof}
  By symmetry, it suffices to show how to pass from \(T_1\) to \(H_1\):  we will argue that
  \(\Tor_i^{\Rep G}(\Rep T_1,\Rep H_2) = 0\)
  implies
  \(\Tor_i^{\Rep G}(\Rep H_1,\Rep H_2) = 0\).
  As in the proof of \cite[Corollary~3.10]{SCSQ} (the corollary quoted as \Cref{cor:proj-dim} above),
  this can be seen using holomorphic induction (\Cref{holomorphic-induction}):
  the inclusion \(i\colon T_1\hookrightarrow H_1\) defines a homomorphism of \(\Rep H_1\)-modules \(i_*\) splitting \(i^*\), and homomorphisms of \(\Rep H_1\)-modules are, in particular, homomorphisms of \(\Rep G\)-modules.  So \(i^*\) induces a split injection of \(\Tor_i^{\Rep G}(\Rep H_1,\Rep H_2)\) into  \(\Tor_i^{\Rep G}(\Rep T_1,\Rep H_2)\), and the claim follows.
  Localization does not affect this argument.
\end{proof}

\section{The proof}

We now proceed to prove \Cref{thm:Tor-of-pairs}, starting from the most special case and generalizing in several steps.
We repeatedly use the fact that \(\Tor\) commutes with localization: for any prime ideal \(\ideal p\subset \Rep G\), we have canonical isomorphisms
\begin{equation}\label{eq:Tor-commutes-with-localization}
  [\Tor_i^{\Rep G}(\Rep H_1, \Rep H_2)]_{\ideal p} \cong   \Tor_i^{(\Rep G)_{\ideal p}}((\Rep H_1)_{\ideal p}, (\Rep H_2)_{\ideal p}).
\end{equation}
Recall from \Cref{supports}~(\labelcref{supports:centre}) that \(\centre{G}\) denotes the centre of~\(G\).

\begin{prop}[particular subgroups of maximal ranks]\label{maximal-rank-particular-subgroups}
  Suppose \(G\) is a good Lie group, and that \(H_1\) and \(H_2\) are connected subgroups whose ranks are maximal in the sense that 
  \begin{equation}\label{eq:max-rank-assumption-H}
    \rank H_1 + \rank H_2 - \interrank{H_1}{H_2} = \rank G.
  \end{equation}
  Assume that \(H_1\) is a torus.  Assume moreover that the restriction \(\Rep G \to \Rep H_2\) is surjective and that its kernel is generated by a regular sequence \(\sequence y = (y_1,\dots,y_{c_2})\) of length \(c_2 = \rank G - \rank H_2\).  Then
  \(
        \Tor_i^{\Rep G}(\Rep H_1, \Rep H_2)_{\ideal p} = 0
        \)
 for all  \(i > 0\) and for all prime ideals \(\ideal p \in \support[\Rep G]{\Rep{(\centre G )}}\).
\end{prop}
\begin{proof}
  Choose a maximal torus \(T\subset G\) containing the torus \(H_1\), and let \(i\colon T\hookrightarrow G\) denote the inclusion.
  Then \(H_1\) is a direct factor of \(T\) \cite{HilgertNeeb}*{Lem.~15.3.2}, and the kernel of the restriction \(\Rep T\to \Rep H_1\) is generated by a regular sequence \(\sequence x = (x_1,\dots,x_{c_1})\) of length \(c_1 = \rank G - \rank H_1\).
  It follows that \(\Rep H_1 \cong \Rep T/(\sequence x)\) as an \(\Rep T\)-module, and that the Koszul complex defined by \(\sequence x\) provides a free resolution of \(\Rep H_1\) as an \(\Rep T\)-module.  We write \(\koszul{\Rep T}{\sequence x}\)  for this complex, and \(\koszul{\Rep T}{\sequence x}\to \Rep H_1\) for the resolution.  As \(\Rep T\) is a free \(\Rep G\)-module, this is also a free resolution of \(\Rep H_1\) as an \(\Rep G\)-module.

  For \(H_2\), we have an isomorphism of \(\Rep G\)-modules \(\Rep H_2 \cong \Rep G/(\sequence y)\) by assumption, and the Koszul complex \(\koszul{\Rep G}{\sequence y}\) provides a free resolution of \(\Rep H_2\) as an \(\Rep G\)-module.

  With these free resolutions \(\koszul{\Rep T}{\sequence x}\) and \(\koszul{\Rep G}{\sequence y}\) of \(\Rep H_1\) and \(\Rep H_2\) at our disposal, we can compute the Tor groups we are interested in as the homology groups of the tensor product of the resolutions (see, for example, \cite[Chapter~VI, above Proposition~1.2]{CartanEilenberg}):

  \begin{equation}\label{eq:Tor-via-double-resolution}
    \Tor_i^{\Rep G}(\Rep H_1, \Rep H_2) \cong \H_i(\koszul{\Rep T}{\sequence x}\otimes_{\Rep G}\koszul{\Rep G}{\sequence y})
  \end{equation}

  Next, observe that we have an isomorphism of \(\Rep G\)-modules, and an isomorphism of complexes of \(\Rep G\)-modules, as follows:
  \begin{align}
    \Rep H_1 \otimes_{\Rep G} \Rep H_2 &\cong \Rep T/(\sequence x, i^*\sequence y)\label{eq:koszul-trick-0}\\
    \koszul{\Rep T}{\sequence x}\otimes_{\Rep G} \koszul{\Rep G}{\sequence y} &\cong \koszul{\Rep T}{\sequence x, i^*\sequence y}\label{eq:koszul-trick}
  \end{align}
  The first isomorphism is clear.  To verify the second isomorphism, note that Koszul complexes can be decomposed into tensor products of Koszul complexes of lengths two.  Specifically, in our case, we have
  \begin{align*}
    \koszul{\Rep G}{\sequence y} &\underset{\Rep G}\cong \;\bigotimes_{\mathclap{\substack{\Rep G\\j=1,\dots,c_2}}} \left(\Rep G \xrightarrow{y_j} \Rep G\right),
    \shortintertext{ and similarly }
    \koszul{\Rep T}{\sequence x, i^*\sequence y} &\underset{\Rep T}\cong \; \koszul{\Rep T}{\sequence x} \;\otimes_{\Rep T}\; \bigotimes_{\mathclap{\substack{\Rep G\\j=1,\dots,c_2}}} \left(\Rep T \xrightarrow{i^* y_j} \Rep T\right).
  \end{align*}
  In view of these decompositions, the isomorphism \eqref{eq:koszul-trick} can be verified by successively applying the isomorphisms
  \begin{align*}
    \left(\Rep T\xrightarrow{i^*y_j} \Rep T\right) &\underset{\Rep G}\cong \Rep T \otimes_{\Rep G} (\Rep G \xrightarrow{y_j} \Rep G)
  \end{align*}
  for \(j = 1, \dots, c_2\).
  We have taken care to distinguish \(\Rep T\)-module isomorphisms from \(\Rep G\)-module isomorphisms in the notation, but note that every isomorphism of \(\Rep T\)-modules is \textit{a fortiori} an isomorphism of \(\Rep G\)-modules.

Combining \eqref{eq:Tor-via-double-resolution} with \eqref{eq:koszul-trick}, we arrive at the following isomorphism of \(\Rep G\)-modules:
\begin{equation}\label{eq:Tor-upshot}
  \Tor_i^{\Rep G}(\Rep H_1, \Rep H_2) \cong \H_i(\koszul{\Rep T}{\sequence x, i^*\sequence y}).
\end{equation}

    The basic strategy now is to compute the Krull dimension of \(\Rep T/(\sequence x, i^*\sequence y)\), to deduce that \((\sequence x, i^*\sequence y)\) is a regular sequence in \(\Rep T\), and to conclude that the augmented complex \(\koszul{\Rep T}{\sequence x, i^*\sequence y}\to \Rep T/(\sequence x, i^*\sequence y)\) is exact, so that the homology groups in \eqref{eq:Tor-upshot} vanish in positive degrees.

    In order to implement this strategy, we now fix and localize at a prime ideal \(\ideal p\) of \(\Rep G\) that lies in \(\support[\Rep G]{\Rep{(\centre G)}}\).
    In view of \cref{eq:Tor-commutes-with-localization}, we may and will moreover assume that \(\ideal p\in \support[\Rep G]{\Rep H_1}\cap \support[\Rep G]{\Rep H_2}\).

    By \Cref{supports}~(\labelcref{supports:centre}), 
    there is a unique ideal \(\ideal p_T\subset \Rep T\) lying over \(\ideal p \subset \Rep G\).
    It follows from the construction that \((\sequence x,i^*\sequence y) \subseteq \ideal p_T\):  We have \((\sequence x) = \ker(\Rep T \to \Rep H_1)\), so \((i^*)^{-1}(\sequence x)\subseteq \ker(\Rep G \to \Rep H_1)\).  Moreover, by the equivalence (b~\(\iff\)~c) of \Cref{prop:segal}, \(\ker(\Rep G \to \Rep H_1)\subseteq \ideal p\).  So \((i^*)^{-1}(\sequence x)\subseteq \ideal p\), and we deduce from \Cref{RG-for-connected-G}~(\labelcref{RG-for-connected-G:max-RT-primes}) that \((\sequence x)\subseteq \ideal p_T\).
    Similarly, \((\sequence y) = \ker(\Rep G \to \Rep H_2) \subseteq \ideal p\), and as \(\ideal p = (i^*)^{-1}(\ideal p_T)\), we find \((i^*\sequence y)\subseteq \ideal p_T\).

    We can therefore view \((\sequence x, i^*\sequence y)\) as a sequence in the maximal ideal \(\ideal p_T\cdot(\Rep T)_{\ideal p}\) of the local ring \((\Rep T)_{\ideal p}\).
    Using \eqref{eq:koszul-trick-0} and the first equality of \Cref{supports}~(\labelcref{supports:intersection}),
    we see that \(\ideal p \in \support[\Rep G]{(\Rep T/(\sequence x, i^*\sequence y))}\). By \Cref{dimension-lemma}, we can find another prime \(\ideal p'\in \support[\Rep G]{(\Rep T/(\sequence x, i^*\sequence y))}\) with \(\ideal p' \subseteq \ideal p\) such that the first equality in the following chain of equalities and inequalities holds:
    \begin{equation}
      \begin{aligned}
        \dim{(\Rep T)_{\ideal p}} - \dim{(\Rep T/(\sequence x,i^*\sequence y))_{\ideal p}}
        &= \dim{(\Rep G)_{\ideal p}} - \dim{(\Rep G/\ideal p')_{\ideal p}}\\
        &= \dim{\Rep G} - \dim{\Rep G/\ideal p'}\\
        &\geq \rank G - \interrank{H_1}{H_2} \\
        &= c_1 + c_2.
      \end{aligned}
      \label{eq:key-dim-calculation}
    \end{equation}
    For the second equality in this chain, we use the fact that the dimension formula holds in \(\Rep G\) (\Cref{RG-biequidimensional-for-connected-G}), and, a fortiori,
    in the localization\footnote{
      For the localization \((\Rep G)_{\ideal p}\), the dimension formula may alternatively be deduced following the argument for \((\Rep T)_{\ideal p}\) presented in the next paragraph:  \((\Rep G)_{\ideal p}\) is a regular local ring, hence Cohen--Macaulay, so a version of the dimension formula for arbitrary, not necessarily prime, ideals holds by \cite[Theorem~17.4\,(i)]{Matsumura:CRT}.
    }
     \((\Rep G)_{\ideal p}\), so that both sides of the equation can be identified with the dimension of \((\Rep G_{\ideal p})_{\ideal p'} = \Rep G_{\ideal p'}\).
For the inequality, note that by \eqref{eq:koszul-trick-0} the ideal \(\ideal p'\) is contained in \(\support[\Rep G]{(\Rep H_1\otimes_{\Rep G} \Rep H_2)}\), and that hence by \Cref{supports}~(\labelcref{supports:intersection}) it is contained in \(\support[\Rep G]{(H_1 \cap gH_2g^{-1})}\) for some \(g\in G\).  The inequality therefore follows from \Cref{supports}~(\labelcref{supports:dim}) and \Cref{RG-for-general-G}~(\labelcref{RG-for-general-G:dimension}).  The final equality is immediate from \eqref{eq:max-rank-assumption-H} and the definitions of \(c_i\) as \(\rank G - \rank H_i\).
Note that \(c_1 + c_2\)  is precisely the length of the sequence \((\sequence x, i^*\sequence y)\).

On the other hand, as \((\Rep T)_{\ideal p}\) is regular (\Cref{RG-regular-for-good-G}), and hence in particular Cohen--Macaulay \cite[Theorem~17.8]{Matsumura:CRT}, we also know that the difference in dimensions on the left side of \eqref{eq:key-dim-calculation} is equal to the height of the ideal generated by \(\sequence x\) and \(i^*\sequence y\)  \cite[Theorem~17.4\,(i)]{Matsumura:CRT}. This height is bounded above by \(c_1 + c_2\), by Krull's Hauptidealsatz.  Thus, we obtain equality in \eqref{eq:key-dim-calculation}, and conclude:
\[
  \dim{(\Rep T)_{\ideal p}} - \dim{(\Rep T)_{\ideal p}/(\sequence x,i^*\sequence y)} = c_1 + c_2
\]
Using once more that \((\Rep T)_{\ideal p}\) is Cohen--Macaulay, we can now apply \cite{Hartshorne}*{Chapter~II, Theorem~8.21A~(c)} or \cite[Theorem~17.4\,(iii)]{Matsumura:CRT} to deduce that the sequence \((\sequence x, i^*\sequence y)\) is regular.
Hence, using \cite[Theorem~16.5\,(i)]{Matsumura:CRT}, we conclude that the Koszul complex \(\koszul{(\Rep T)_{\ideal p}}{\sequence x, i^*\sequence y}\) is exact in all positive degrees.
\end{proof}

\begin{prop}[tori of maximal ranks]\label{maximal-rank-tori}
  Suppose \(G\) is a good Lie group, and that \(T_1\) and \(T_2\) are tori in \(G\)
  such that
  \begin{equation}\label{eq:max-rank-assumption-T}
    \rank T_1 + \rank T_2 - \interrank{T_1}{T_2} = \rank G.
  \end{equation}
  Then
  \(
  \Tor_i^{\Rep G}(\Rep T_1, \Rep T_2)_{\ideal p} = 0
  \)
  for all  \(i > 0\) and for all prime ideals \(\ideal p \in \support[\Rep G]{\Rep{(\centre G )}}\).
\end{prop}
\begin{proof}
  Consider the triple \((G\times G, T_1\times T_2, \Delta G)\).
  The group \(G\times G\) is good as \(G\) is good and \(\Rep (G\times G)\cong \Rep{G}\otimes\Rep{G}\).
  The subgroups \(T_1\times T_2\) and \(\Delta G\) are connected.
  Using \Cref{diagonal-reformulation}~(\labelcref{diagonal-reformulation:interrank}), we see that the triple still satisfies the maximal rank condition \eqref{eq:max-rank-assumption-H}.  The restriction \(\Delta^*\colon \Rep(G\times G)\to \Rep{G}\) is surjective.  Also, given that \(G\) is good, we can easily find a regular sequence of length \(\rank G\) generating the kernel of \(\Delta^*\).  So \Cref{maximal-rank-particular-subgroups} applies to \((G\times G, T_1\times T_2, \Delta G)\).  Using \Cref{diagonal-reformulation}~(\labelcref{diagonal-reformulation:Tor}) and (\labelcref{diagonal-reformulation:support}), we conclude that
  \(
  \Tor_i^{\Rep G}(\Rep T_1, \Rep T_2)_{\ideal p} = 0
  \)
  for all  \(i > 0\) and for all prime ideals \(\ideal p \in \support[\Rep G]{\Rep{(\centre G )}}\).
\end{proof}

\begin{prop}[arbitrary subgroups]\label{arbitrary-subgroups}
  Suppose \(G\) is a good Lie group, and that \(H_1\) and \(H_2\) are connected subgroups.
  Then
  \(
  \Tor_i^{\Rep G}(\Rep H_1, \Rep H_2)_{\ideal p} = 0
  \)
  for all
  \[
    i >  \rank G - (\rank H_1 + \rank H_2) + \interrank{H_1}{H_2}
  \]
  and for all prime ideals \(\ideal p \in \support[\Rep G]{\Rep{(\centre G )}}\).
\end{prop}
\begin{proof}
  Choose maximal tori \(T_1\) and \(T_2\) of \(H_1\) and \(H_2\), respectively.
  By \Cref{biquotient-condition-can-be-checked-on-tori}, passing from the groups \(H_i\) to their maximal tori \(T_i\) does not change the intersection rank.
  By \Cref{enlarging-tori}, we can find a torus \(T_1^+\supseteq T_1\) such that \(\interrank{T_1^+}{T_2} = \interrank{T_1}{T_2}\) and such that the triple \((G,T_1^+,T_2)\) satisfies the maximal rank condition \eqref{eq:max-rank-assumption-T}.
  So by \Cref{maximal-rank-tori}, \(\Tor_*^{\Rep G}(\Rep{(T_1^+)}, \Rep T_2)_{\ideal p}\) vanishes in all positive degrees for all \(\ideal p \in \support[\Rep G]{\Rep{(\centre G )}}\).

  We can now apply \Cref{Tor-CoR-I} with \(\Lambda\), \(\Gamma\), \(A\) and \(C\) as follows:
  \[
    \begin{tikzcd}[column sep=tiny,row sep=tiny]
      (\Rep T_2)_{\ideal p} \arrow[d,dash] & (\Rep T_1)_{\ideal p} \arrow[d,dash]\\
      (\Rep G)_{\ideal p} \ar[r] & (\Rep T_1^+)_{\ideal p}
    \end{tikzcd}
  \]
  This gives us isomorphisms
  \[
    \Tor_i^{\Rep G}(\Rep T_2, \Rep T_1)_{\ideal p} \cong \Tor_i^{\Rep T_1^+}(M,\Rep T_1)_{\ideal p}
  \]
  in all degrees~\(i\), with \(M := \Rep{T_2}\otimes_{\Rep G}\Rep{T_1^+}\). The Tor groups on the right vanish for all \(i > \projdim_{\Rep T_1^+} \Rep T_1\), and by \Cref{cor:proj-dim}, \(\projdim_{\Rep T_1^+} \Rep T_1  \leq \rank T_1^+ - \rank T_1\).  As  \((G,T_1^+,T_2)\) satisfies \eqref{eq:max-rank-assumption-T}, and as \(\interrank{T_1^+}{T_2} = \interrank{T_1}{T_2}\), this bound can be rewritten as \(\rank T_1^+ - \rank T_1 = \rank G - \rank T_1 - \rank T_2 + \interrank{T_1}{T_2}\).
  We thus deduce that
  \(  \Tor_i^{\Rep G}(\Rep T_1, \Rep T_2)_{\ideal p}
  \)
  vanishes in the required range.   Using \Cref{red:torus}, we obtain the same claim for the triple \((G,H_1,H_2)\).
\end{proof}

\Cref{thm:Tor-of-pairs} is now just a special case:

\begin{proof}[Proof of \Cref{thm:Tor-of-pairs}]
  Fix some degree \(i > \rank G - (\rank H_1 + \rank H_2) + \interrank{H_1}{H_2}\).
  Consider a prime \(\ideal p\) supporting \(\Tor_i^{\Rep G}(\Rep H_1, \Rep H_2)\).  Any such prime necessarily lies in the support of both \(\Rep{H_1}\) and \(\Rep{H_2}\).
  As \((G,H_1,H_2)\) satisfies the lax biquotient condition, this implies that
  \(\ideal p\) also lies in the support of \(\Rep{(\centre G)}\), as noted in \Cref{eg:lax-biquotient-condition-translated-into-commutative-algebra}.
  Hence by \Cref{arbitrary-subgroups}, \(\Tor_i^{\Rep G}(\Rep H_1, \Rep H_2)_{\ideal p} = 0\).
  So we find that the localization of \(\Tor_i^{\Rep G}(\Rep H_1, \Rep H_2)\) at \emph{any} prime ideal of \(\Rep{G}\) is zero, hence \(\Tor_i^{\Rep G}(\Rep H_1, \Rep H_2)\) itself must vanish.
\end{proof}

Finally, \Cref{thm:Tor-of-pairs-simplified} is, of course, a special case of \Cref{thm:Tor-of-pairs}.

\section{K-theory of biquotients}
\label{sec:K}
This short final section explains how to apply \Cref{thm:Tor-of-pairs-simplified} to biquotients to obtain \Cref{K-of-biquotients}.
\newcommand{\leftcs}[2]{#1\backslash #2}%
\newcommand{\rightcs}[2]{#1/ #2}%

\begin{thm}[\citelist{\cite{Hodgkin}\cite{McLeod}*{Thm~4.1}}]
  Suppose \(G\) is a compact connected Lie group with torsion-free fundamental group. Then for \(G\)-spaces \(X\) and \(Y\) we have a strongly convergent spectral sequence
  \[
    \{\Tor_{-p}^{\Rep G}(\K_G^*(X),\K_G^*(Y))\}^q
    \Rightarrow \K^{p+q}_G(X\times Y).
  \]
  Here, the \(\Rep{G}\)-modules \(\K_G^*(X)\) and \(\K_G^*(Y)\) are understood to be \(\ZZ/2\)-graded, and \(\Tor\) denotes graded \(\Tor\).
  The \(G\)-spaces are assumed to be compactly generated and have the homotopy type of a CW complex; see \cite[\S\,I.1]{Hodgkin}.
\end{thm}

Given a strict biquotient \(\biquotient{G}{H_1}{H_2}\), we have a principal \(G\)-bundle \(\leftcs{H_1}{G}\times\rightcs{G}{H_2} \to \biquotient{G}{H_1}{H_2}\) \cite{Singhof:DCM}. Take
\( X := \leftcs{H_1}{G}\) and
\(Y := \rightcs{G}{H_2}\).
Then \(\K^*_G(X) \cong \Rep H_1\), \(\K^*_G(Y) \cong \Rep H_2\), and \(\K^*_G(X\times Y) \cong \K^*(\biquotient{G}{H_1}{H_2})\).
So the spectral sequence takes the form
\[
  \Tor_{-p}^{\Rep G}(\Rep H_1,\Rep H_2)
  \Rightarrow \K^p(\biquotient{G}{H_1}{H_2}).
\]
In particular, given that \(\Rep H_1\) and \(\Rep H_2\) are concentrated in degree zero, the graded Tor is just the usual Tor.
\Cref{thm:Tor-of-pairs-simplified} therefore implies the results on the K-theory of biquotients summarized in \Cref{K-of-biquotients} in exactly the same way that Steinberg's Theorem and more generally \Cref{cor:proj-dim} provide a computation of the K-theory of homogeneous spaces.  We refer to \cite[proof of Theorem~3.6 and, in particular, Figure~1]{SCSQ} for details of the argument.

\end{document}